\documentclass[article,10pt]{article}
\usepackage{amsmath}
\usepackage{amsfonts}
\usepackage{array}
\usepackage{enumerate}
\usepackage{graphicx}
\usepackage{amssymb}
\usepackage{amsthm}
\usepackage{xcolor}
\usepackage{chemarr}
\usepackage[left=2.5cm, right=2.5cm]{geometry}
\usepackage{ulem}

\usepackage{caption}
\usepackage{booktabs}
\usepackage{makecell}
\usepackage{siunitx}
\usepackage{tcolorbox}

\newtheorem{proposition}{Proposition}
\newtheorem{lemma}{Lemma}

\newtheorem{remark}{Remark}

\newcommand{\norm}[2][\relax]{\ifx#1\relax \ensuremath{\left\Vert#2\right\Vert} \else \ensuremath{\left\Vert#2\right\Vert_{#1}}\fi}

\begin{document}

\title{The Michaelis--Menten reaction at low substrate concentrations: \\
Pseudo-first-order kinetics and conditions for timescale separation}

\author{Justin Eilertsen\\
            Mathematical Reviews\\ American Mathematical Society\\
            416 4th Street\\Ann Arbor, MI, 48103\\
            e-mail: {\tt jse@ams.org}\\\\
        Santiago Schnell\\
            Department of Biological Sciences and\\
            Department of Applied and Computational Mathematics and Statistics\\
            University of Notre Dame\\
            Notre Dame, IN 46556\\
            e-mail: {\tt santiago.schnell@nd.edu}\\\\
        Sebastian Walcher\\
            Mathematik A, RWTH Aachen\\
            D-52056 Aachen, Germany\\
            e-mail: {\tt walcher@matha.rwth-aachen.de}
        }


\maketitle
\begin{abstract} 
We demonstrate that the Michaelis--Menten reaction mechanism can be accurately approximated by a linear system when the initial 
substrate concentration is low. This leads to pseudo-first-order kinetics, simplifying mathematical calculations and experimental
analysis. Our proof utilizes a monotonicity property of the system and Kamke's comparison theorem. This linear approximation yields 
a closed-form solution, enabling accurate modeling and estimation of reaction rate constants even 
without timescale separation.  Building on prior work, we establish that the sufficient condition for the validity of this 
approximation is $s_0 \ll K$, where $K=k_2/k_1$ is the Van Slyke--Cullen constant. This condition is independent of the initial 
enzyme concentration. Further, we investigate timescale separation within the linear system, identifying necessary and sufficient 
conditions and deriving the corresponding reduced one-dimensional equations.\\

\end{abstract}

\section{Introduction and overview of results} 
The irreversible Michaelis--Menten reaction mechanism describes the mechanism of action of a fundamental reaction for biochemistry. 
Its governing equations are:
\begin{equation}\label{eqmmirrev}
\begin{array}{rccrclcl}
\dot s&=&f_1(s,c):=& -k_1e_0s&+&(k_1s+k_{-1})c & &  \\
\dot c&=& f_2(s,c):=&k_1e_0s&-&(k_1s+k_{-1}+k_2)c & &\\
\end{array}
\end{equation}
with positive parameters $e_0$, $s_0$, $k_1$, $k_{-1}$, $k_2$, and typical initial conditions $s(0)=s_0$, $c(0)=0$. Understanding 
the behavior of this system continues to be crucial for both theoretical and experimental purposes, such as identifying the rate 
constants $k_1$, $k_{-1}$ and $k_2$. However, the Michaelis--Menten system~\eqref{eqmmlin} cannot be solved via elementary 
functions, hence approximate simplifications have been proposed since the early 1900's.

In 1913, Michaelis and Menten~\cite{MM1913} introduced the partial equilibrium assumption for enzyme-substrate complex formation, 
effectively requiring a small value for the catalytic rate constant $k_2$. Later, for systems with low initial enzyme concentration 
($e_0$), Briggs and Haldane~\cite{BH} derived the now-standard quasi-steady-state reduction (see Segel and Slemrod~\cite{SSl} for 
a systematic analysis). Both these approaches yield an asymptotic reduction to a one-dimensional equation. A mathematical justification is obtained from
singular perturbation theory. Indeed, most literature on the Michaelis--Menten system focuses on dimensionality reduction, with 
extensive discussions on parameter combinations ensuring this outcome.

We consider a different scenario: low initial substrate concentration ($s_0$). This presents a distinct case from both the low-enzyme 
and partial-equilibrium scenarios, which yield one-dimensional reductions in appropriate parameter ranges. Instead, with 
low substrate, we obtain a different simplification:
\begin{equation}\label{eqmmlin}
    \begin{pmatrix}
        \dot s \\ \dot c
    \end{pmatrix}=\begin{pmatrix}
    -k_1e_0 & k_{-1}\\
    k_1e_0& -(k_{-1}+k_2)
\end{pmatrix}\cdot\begin{pmatrix}
        s\\ c
    \end{pmatrix}.
\end{equation}
This linear Michaelis--Menten system~\eqref{eqmmlin} was first studied by Kasserra and Laidler~\cite{Laidler}, who proposed the 
condition of excess initial enzyme concentration ($e_0 \gg s_0$) for the validity of the linearization of \eqref{eqmmirrev} to
\eqref{eqmmlin}.  This aligns with the concept of pseudo-first-order kinetics in chemistry, where the concentration of one 
reactant is so abundant that it remains essentially constant \cite{Silicio}.

Pettersson~\cite{Pettersson} further investigated the linearization, adding the assumption that any complex concentration accumulated 
during the transient period remains too small to significantly impact enzyme or product concentrations. Later, Schnell and 
Mendoza~\cite{Schnell} refined the validity condition $s_0 \ll K_M$ (where $K_M = (k_{-1}+k_2)/k_1$ is the Michaelis constant) 
to be sufficient for the linearization (alias dictus the application of pseudo-first-order kinetics to the Michaelis--Menten reaction). 
However, Pedersen and Bersani~\cite{Pedersen} found this condition overly conservative. They introduced the total substrate 
concentration $\bar s=s+c$ (used in the total quasi-steady state approximation \cite{BBS}) and proposed the condition $s_0 \ll K_M + e_0$, 
reconciling the Schnell--Mendoza and Kasserra--Laidler conditions.

In this paper, we demonstrate that the solutions of the Michaelis--Menten system~\eqref{eqmmirrev} admit a global approximation by solutions 
of the linear differential equation~\eqref{eqmmlin} whenever the intial substrate concentration is small. Our proof utilizes Kamke's comparison theorem for cooperative differential 
equations.\footnote{This concept is distinct from cooperativity in biochemistry.} This theorem allows us to establish upper and lower 
estimates for the solution components of \eqref{eqmmirrev} via suitable linear systems. Importantly, these bounds are valid whenever
$s_0<K$, where $K=k_2/k_1$ is the Van Slyke--Cullen constant, and the estimates are tight whenever $s_0\ll K$. This effectively serves as sufficient condition for the validity 
of the linear approximation Michaelis--Menten system~\eqref{eqmmlin}.  We further prove that solutions of the original Michaelis--Menten
system~\eqref{eqmmirrev} converge to solutions of the linear system~\eqref{eqmmlin} as $s_0\to 0$, uniformly for all positive times. Moreover,
the asymptotic rate of convergence is of order $s_0^2$.

Additionally, we explore the problem of of timescale separation within the linear Michaelis--Menten system~\eqref{eqmmlin}. 
While a linear approximation simplifies the system, it does not guarantee a clear separation of timescales.  In fact, there are
cases where no universal accurate one-dimensional approximation exists. Since both eigenvalues of the matrix in \eqref{eqmmlin} are real 
and negative, the slower timescale will dominate the solution's behavior.  True timescale separation occurs if and only if these 
eigenvalues differ significantly in magnitude. Importantly, we find that this separation is not always guaranteed. There exist 
parameter combinations where a one-dimensional approximation would lack sufficient accuracy. Building upon results from \cite{ESWAnti}, 
we establish necessary and sufficient conditions for timescale separation and derive the corresponding reduced one-dimensional 
equations.

\section{Low substrate: Approximation by a linear system}
To set the stage, we recall some facts from the theory of cooperative differential equation systems. For a comprehensive account 
of the theory, we refer to the H.L.~Smith's monograph \cite[Chapter~3]{Smi} on monotone dynamical systems. Consider the standard 
ordering, here denoted by $\prec$, on $\mathbb R^n$, given by
\begin{equation*}
\begin{pmatrix}
x_1\\ \vdots \\ x_n 
\end{pmatrix} \prec \begin{pmatrix}
y_1\\ \vdots \\y_n 
\end{pmatrix} \Longleftrightarrow x_i\leq y_i\text{  for  } 1\leq i\leq n.
\end{equation*}

Now let a differential equation
\begin{equation}\label{ode}
\dot x = f(x)\quad \text{on  } U\subseteq \mathbb R^n,\quad \emptyset \not= U\text{  open  }
\end{equation}
be given, with $f$ continuously differentiable. We denote by $F(t,y)$ the solution of the initial value problem \eqref{ode} with 
$x(0)=y$, and call $F$ the \textit{local flow} of the system.

Given a convex subset $D\subseteq U$, we say that \eqref{ode} is \textit{cooperative on $D$} if all non-diagonal elements of the 
Jacobian are nonnegative; symbolically
\begin{equation*}
Df(x)=\begin{pmatrix} *& +&\cdots& &+ \\ +&*&+&\cdots&+\\  \vdots&+&\ddots& & \vdots\\ \vdots &  &  &\ddots & + \\+&\cdots & \cdots & + & *
\end{pmatrix} \text{  for all  } x\in D.
\end{equation*}

As shown by M.W.~Hirsch (\cite[and subsequent papers]{Hir}) and other researchers (see \cite[and the references therein]{Smi}), 
cooperative systems have special qualitative features. Qualitative properties of certain monotone chemical reaction networks 
were investigated by De Leenheer et al.~\cite{DLAS}. In the present paper, we consider a different feature of such systems.
Our interest lies in the following simplified version of Kamke's comparison theorem \cite{Kam}.\footnote{Kamke \cite{Kam} 
considered non-autonomous systems with less restrictive properties. In particular the statement also holds on P-convex sets. 
A subset $D$ of $\mathbb R^n$ is called P-convex if, for each pair $y,\,z\in D$ with $y\prec z$, the line segment connecting 
$y$ and $z$ is contained in $D$. Thus, P-convexity is a weaker property than convexity.}

\begin{proposition} Let $D\subseteq U$ be a convex positively invariant set for \eqref{ode}, and assume that this system is 
cooperative on $D$. Moreover let $\dot x=g(x)$ be defined on $U$, with continuously differentiable right hand side, and local 
flow $G$.
\begin{itemize}
\item If $f(x)\prec g(x)$ for all $x\in D$, and $y\prec z$, then $F(t,y)\prec G(t,z)$ for all $t\geq 0$ such that both solutions exist.
\item If $g(x)\prec f(x)$ for all $x\in D$, and $z\prec y$, then $G(t,z)\prec F(t,y)$ for all $t\geq 0$ such that both solutions exist.
\end{itemize}
\end{proposition}

\subsection{Application to the Michaelis--Menten systems~\eqref{eqmmirrev} and \eqref{eqmmlin}}
For the following discussions, we recall the relevant derived parameters
\begin{equation}\label{constants}
K_S:=k_{-1}/k_1,\quad K_M:=(k_{-1}+k_2)/k_1\quad K:=k_2/k_1 =K_M-K_S,
\end{equation}
where $K_S$ is the complex equilibrium constant, $K_M$ is the Michaelis constant, and $K$ is the van Slyke-Cullen constant.

The proof of the following facts is obvious.
\begin{lemma}\label{estlem}
\begin{enumerate}[(a)]
\item The Jacobian of system \eqref{eqmmirrev}
is equal to
    \[
    \begin{pmatrix}
        \ast & k_1s+k_{-1}\\
        k_1(e_0-c) &\ast
    \end{pmatrix}.
    \]
The off-diagonal entries are $\geq 0$ in the physically relevant  subset
\[
D=\left\{(s,c);\, s\geq 0,\, e_0\geq c\geq 0,\, s+c\leq s_0\right\}
\]
of the 
phase space, and $D$ is positively invariant and convex. Thus system \eqref{eqmmirrev} is cooperative in $D$.
\item In $D$, one has 
    \[
    \begin{array}{rcccl}
         f_1(s,c)&\geq &g_1(s,c)&:=&-k_1e_0s+k_{-1}c\\
         f_2(s,c)&\geq &g_2(s,c)&:=&k_1e_0s-(k_1s_0+k_{-1}+k_2)c,
    \end{array}
    \]
hence by Kamke's comparison theorem the solution of the linear system with matrix
    \[
G:=\begin{pmatrix}
    -k_1e_0 & k_{-1}\\
    k_1e_0& -(k_1s_0+k_{-1}+k_2)
\end{pmatrix} = k_1\begin{pmatrix}
    -e_0 & K_S\\
    e_0& -K_M(1+s_0/K_M)
\end{pmatrix} 
    \]
and initial value in $D$ provides component-wise lower estimates for the solution of \eqref{eqmmirrev} with the same 
initial value.
\item In $D$, one has 
    \[
    \begin{array}{rcccl}
         f_1(s,c)&\leq &h_1(s,c)&:=&-k_1e_0s+(k_1s_0+k_{-1})c\\
         f_2(s,c)&\leq &h_2(s,c)&:=&k_1e_0s-(k_{-1}+k_2)c,
    \end{array}
    \]
hence by Kamke's comparison theorem the solution of the linear system with matrix
    \[
H:=\begin{pmatrix}
    -k_1e_0 & k_1s_0+k_{-1}\\
    k_1e_0& -(k_{-1}+k_2)
\end{pmatrix}=k_1\begin{pmatrix}
    -e_0 & K_S(1+s_0/K_S)\\
    e_0& -K_M
\end{pmatrix}
    \]
and initial value in $D$ provides component-wise upper estimates for the solution of \eqref{eqmmirrev} with the same 
initial value.
\end{enumerate}
\end{lemma}

\begin{remark}
The upper estimates become useless in the case $k_1s_0> k_2$, because then one eigenvalue 
of $H$ becomes positive. For viable upper and lower bounds one requires that $s_0< K=k_2/k_1$.
\end{remark}

The right hand sides of both linear comparison systems from parts (b) and (c) of the Lemma converge to 
\begin{equation}\label{star}
Df(0)=
    \begin{pmatrix}
    -k_1e_0 & k_{-1}\\
    k_1e_0& -(k_{-1}+k_2)
\end{pmatrix} =  k_1\begin{pmatrix}
    -e_0 & K_S\\
    e_0& -K_M
\end{pmatrix}
\end{equation}
as $s_0\to 0$. We will use this observation to show that a solution of \eqref{eqmmirrev} with initial value in $D$ 
converges to the solution of \eqref{eqmmlin} with the same initial value on the time interval $[0,\,\infty)$, as 
$s_0\to 0$.

\begin{figure}[htb!]
\centering
\includegraphics[width=8.0cm]{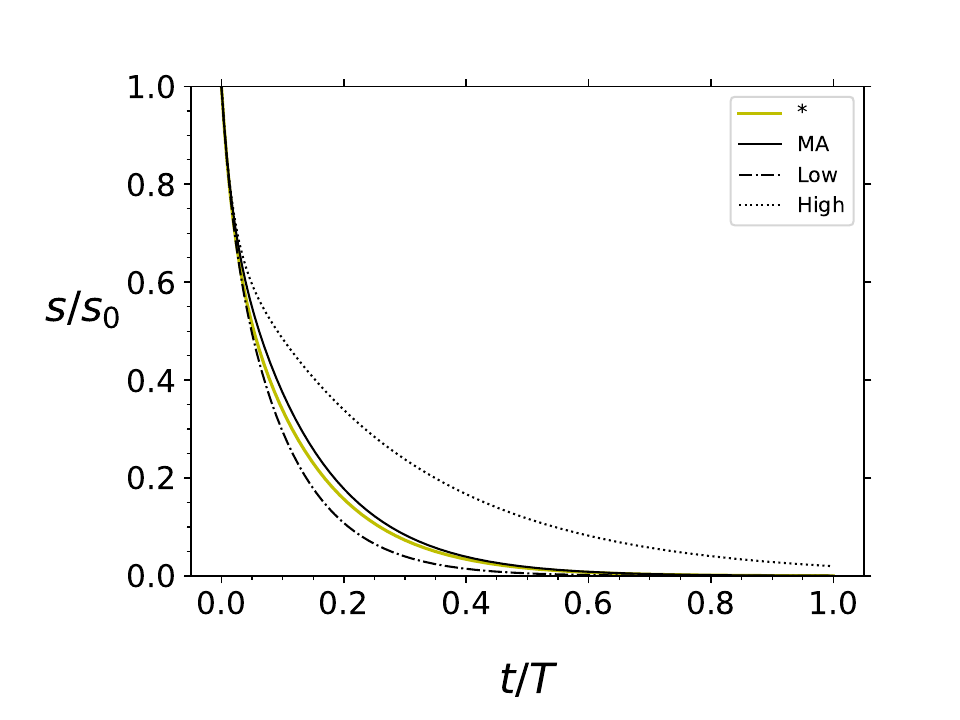}
\includegraphics[width=8.0cm]{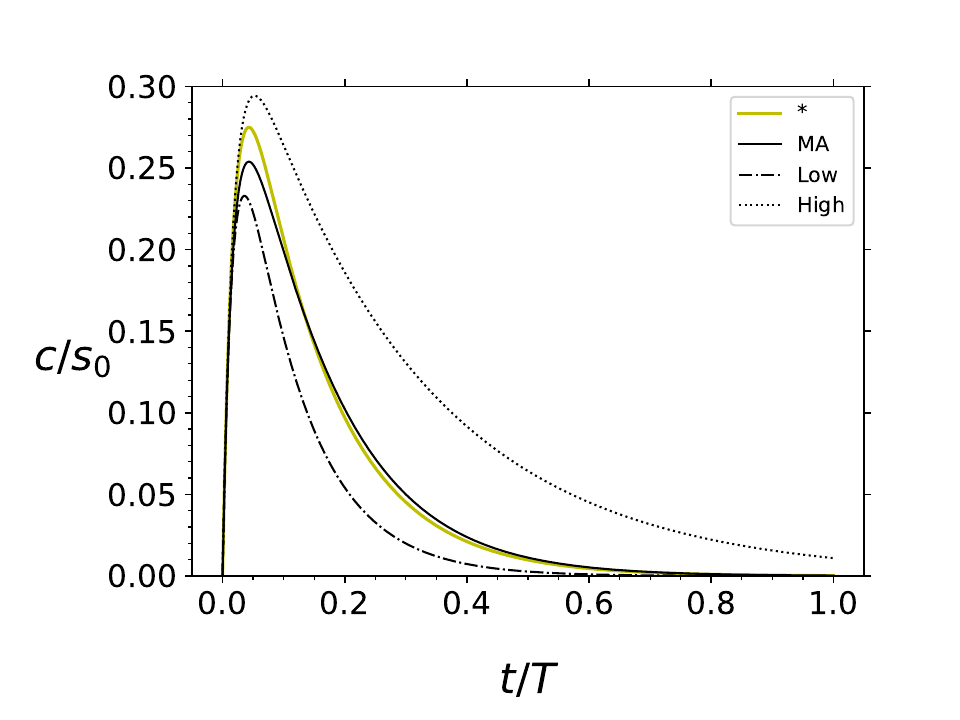}\\
\includegraphics[width=8.0cm]{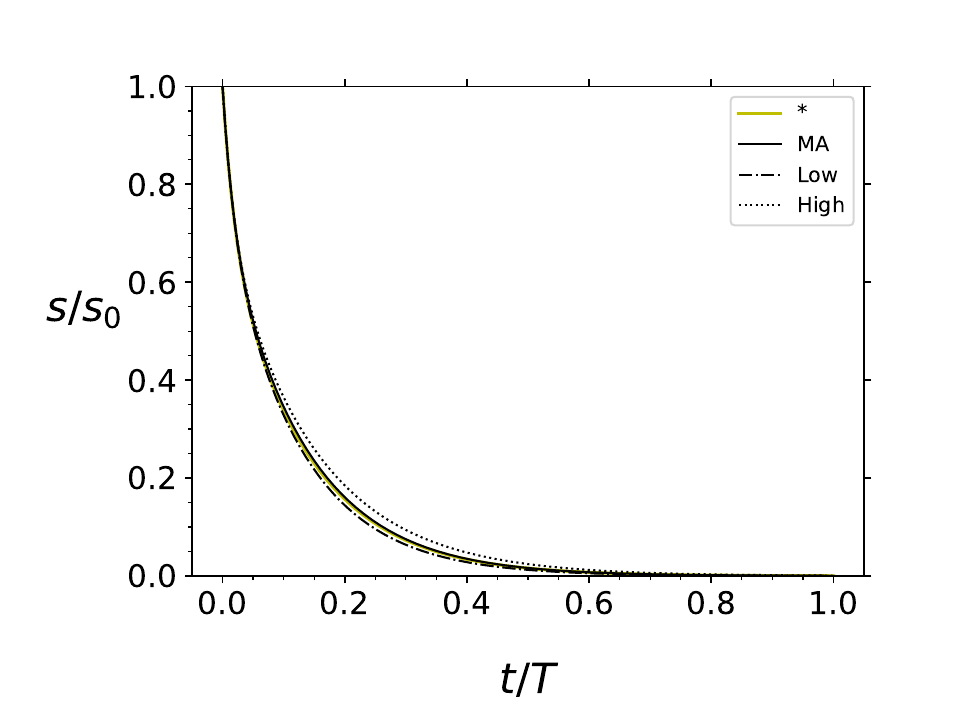}
\includegraphics[width=8.0cm]{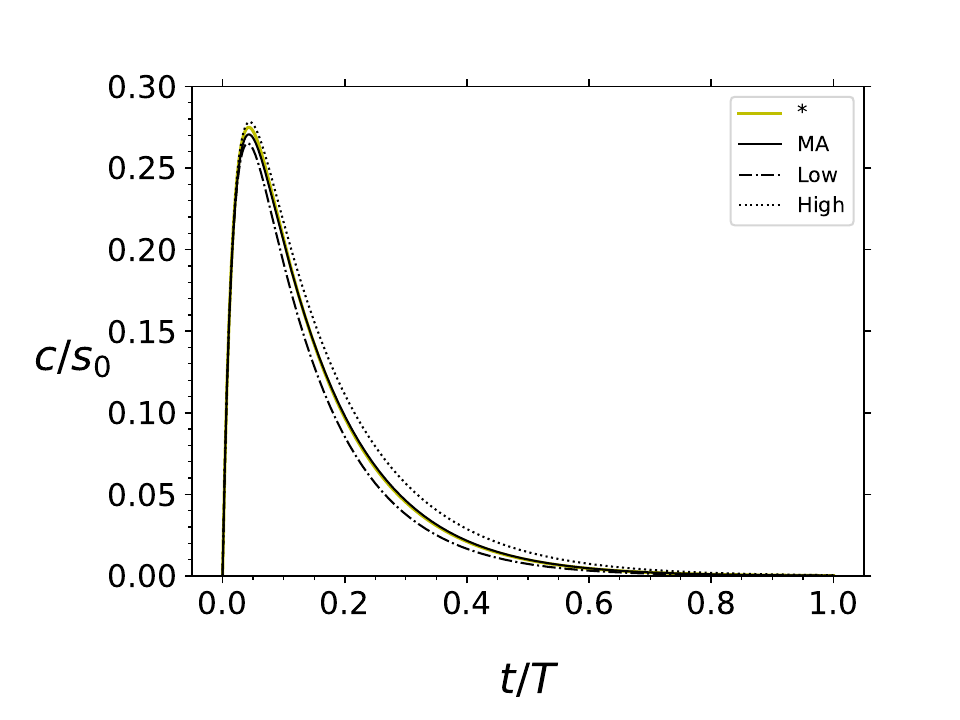}
\caption{{\textbf{Illustration of Proposition \ref{baseprop}. The solution to the Michaelis--Menten 
system~(\ref{eqmmirrev}) converges to the solution of the linear Michaelis--Menten system~(\ref{eqmmlin}) 
as $s_0\to 0$.} In all panels, the solid black curve is the numerical solution to the Michaelis--Menten 
system~(\ref{eqmmirrev}). The thick yellow curve is the numerical solution to linear Michaelis--Menten 
system~(\ref{eqmmlin}). The dashed/dotted curve is the numerical solution to the linear system defined 
by matrix $G$. The dotted curve is the numerical solution to the linear system defined by matrix $H$. 
All numerical simulations where carried out with the following parameters (in arbitrary units): 
$k_1=k_2=k_{-1}=e_0=1$. In all panels, the solutions have been numerically-integrated 
over the domain $t\in[0,T],$ where $T$ is selected to be long enough to ensure that the long-time dynamics 
are sufficiently captured. For illustrative purposes, the horizontal axis (in all four panels) has been 
scaled by $T$ so that the scaled time, $t/T$, assumes values in the unit interval: $\frac{t}{T}\in[0,1]$. 
{\sc Top Left}: The numerically-obtained time course of $s$ with $s_0=0.5$ and $c(0)=0.0$. 
{\sc Top Right}: The numerically-obtained time course of $c$ with $s_0=0.5$ and $c(0)=0.0$. 
{\sc Bottom Left}: The numerically-obtained time course of $s$ with $s_0=0.1$ and $c(0)=0.0$. 
{\sc Bottom Right}: The numerically-obtained time course of $c$ with $s_0=0.1$ and $c(0)=0.0$. 
Observe that the solution components of (\ref{eqmmlin}) become increasingly accurate approximations to 
the solution components of (\ref{eqmmirrev}) as $s_0$ decreases.}}\label{FIG1}
\end{figure}

\subsection{Comparison Michaelis--Menten systems~\eqref{eqmmirrev} and \eqref{eqmmlin}}
Both $G$ and $H$, as well as $Df(0)$, are of the type
\begin{equation}\label{placeholder}
\begin{pmatrix}
    -\alpha & \beta\\  \alpha& -\gamma
\end{pmatrix}
\end{equation}
with all entries $>0$, and $\gamma>\beta$ (to ensure usable estimates). The trace equals $-(\alpha+\gamma)$, 
the determinant  equals $\alpha(\gamma-\beta)$, and the discriminant is 
\[
\Delta=(\alpha+\gamma)^2-4\alpha(\gamma-\beta)=(\alpha-\gamma)^2+4\alpha \beta.
\]
The eigenvalues are 
\[
\lambda_{1,2}=\frac12\left(-(\alpha+\gamma)\pm\sqrt{\Delta}\right)
\]
with eigenvectors
\[
v_{1,2}=\begin{pmatrix}
   \beta \\ \frac12(\alpha-\gamma\pm \sqrt{\Delta}) .
\end{pmatrix}
\]
The following is now obtained in a straightforward (if slightly tedious) manner. We specialize the initial 
value to the typical case, with no complex present at $t=0$.
\begin{lemma}\label{snore}
The solution of the linear differential equation with matrix \eqref{placeholder} and initial value $\begin{pmatrix}
    s_0\\0
\end{pmatrix}$ is equal to
\[
\begin{array}{rccl}
    \begin{pmatrix}
        \widetilde s\\ \widetilde c
    \end{pmatrix} &=& \dfrac{s_0}{2\beta\sqrt{\Delta}}&\left[ (\alpha-\gamma-\sqrt{\Delta})\begin{pmatrix}
       \beta \\ \frac12(\alpha-\gamma+\sqrt{\Delta})
    \end{pmatrix} \,\exp(\left(\frac12 (-\alpha-\gamma+\sqrt{\Delta})t\right)\right.\\
     & & + &\left. (-\alpha+\gamma-\sqrt{\Delta})\begin{pmatrix}
       \beta \\ \frac12(\alpha-\gamma-\sqrt{\Delta})
    \end{pmatrix} \,\exp(\left(\frac12 (-\alpha-\gamma-\sqrt{\Delta})t\right)\right]\\
    &=& \dfrac{s_0}{2\sqrt{\Delta}}&\left[\begin{pmatrix}
        -\alpha+\gamma+\sqrt{\Delta}\\2\alpha
    \end{pmatrix} \,\exp(\left(\frac12 (-\alpha-\gamma+\sqrt{\Delta})t\right)\right.\\
     & & + &\left. \begin{pmatrix}
       \alpha -\gamma+\sqrt{\Delta}\\-2\alpha
    \end{pmatrix}  \,\exp(\left(\frac12(-\alpha-\gamma-\sqrt{\Delta})t\right)\right]\\
\end{array}
\]
    
\end{lemma}
This enables us to write down explicitly the solution of \eqref{eqmmlin}, as well as upper and lower estimates, 
in terms of familiar constants. We  use {\bf Lemma~\ref{snore}} with $\alpha=k_1e_0$, $\beta=k_1K_S$ and 
$\gamma=k_1K_M$, so $\Delta=k_1^2\left((K_M-e_0)^2+4K_Se_0\right)$, to obtain

\begin{equation}\label{nicelinsol}
    \begin{array}{rcl}
    \begin{pmatrix}
        s^*\\ c^*
    \end{pmatrix}     &=& \dfrac{s_0\exp(-\frac12((K_M+e_0)-\sqrt{(K_M-e_0)^2+4K_Se_0}) k_1t)}{2\sqrt{(K_M-e_0)^2+4K_Se_0}}\times \\ 
        & & \qquad\qquad\qquad
        \begin{pmatrix} +(K_M-e_0)+\sqrt{(K_M-e_0)^2+4K_Se_0}\\2e_0
        \end{pmatrix}  \\
         & +&\dfrac{s_0\exp(-\frac12((K_M+e_0)+\sqrt{(K_M -e_0)^2+4K_Se_0}) k_1t)}{2\sqrt{(K_M-e_0)^2+4K_Se_0}}\times \\ 
      & &  \qquad \qquad\qquad\begin{pmatrix}
        -(K_M-e_0)+\sqrt{(K_M -e_0)^2+4K_Se_0}\\-2e_0
    \end{pmatrix} \\
    \end{array}
\end{equation}
as the solution of \eqref{eqmmlin}, after some simplifications. 
\begin{proposition}\label{baseprop}
To summarize:
\begin{enumerate}[(a)]
    \item The solution $\begin{pmatrix}
        s^*\\ c^*
    \end{pmatrix}$ of the linear differential equation \eqref{eqmmlin} with initial value  $\begin{pmatrix}
s_0\\0
\end{pmatrix}$ is given by \eqref{nicelinsol}.
\item The solution 
    $\begin{pmatrix}
        s_{\rm low}\\ c_{\rm low}
    \end{pmatrix}$ 
    of the linear differential equation with matrix $G$ and initial value 
    $\begin{pmatrix}
    s_0\\0
    \end{pmatrix}$ 
is obtained from replacing $K_M$ by $K_M+s_0$ in \eqref{nicelinsol}.
\item The solution 
    $\begin{pmatrix}
        s_{\rm up}\\ c_{\rm up}
    \end{pmatrix}$ 
of the linear differential equation with matrix $H$ and initial value  
    $\begin{pmatrix}
    s_0\\0
    \end{pmatrix}$ 
is obtained from replacing $K_S$ by $K_S+s_0$ in \eqref{nicelinsol}.
\item Given that $s_0<K$, for all $t\geq 0$ one has the inequalities
    \begin{equation}\label{niceests}
        \begin{array}{cc}
          s_{\rm up} \geq s\geq s_{\rm low},  & \quad   s_{\rm up} \geq s^*\geq s_{\rm low}, \\
          c_{\rm up} \geq c\geq c_{\rm low},  & \quad   c_{\rm up} \geq c^*\geq c_{\rm low}; \\
        \end{array}
    \end{equation}
    where 
    $\begin{pmatrix}
        s\\ c
    \end{pmatrix}$ 
    denotes the solution of \eqref{eqmmirrev} with initial value 
    $\begin{pmatrix}
    s_0\\0
\end{pmatrix}$.
\item Let $\mathcal M$ be a compact subset of the open positive orthant $\mathbb R^4_{>0}$, abbreviate 
$\pi:=(e_0, k_1, k_{-1}, k_2)^{\rm tr}$ and $K^*:=\min\{k_2/k_1;\, \pi \in\mathcal{M}\}$. Then, there 
exists a dimensional constant $C$ (with dimension concentration$^{-1}$), depending only on $\mathcal{M}$, 
such that for all $\pi\in{\mathcal M}$ and for all $s_0$ with $0<s_0\leq K^*/2$ the estimates
    \begin{equation*}
        \begin{array}{rcl}
           \left|\dfrac{s-s^*}{s_0} \right| &  \leq& C\cdot s_0\\
              \left|\dfrac{c-c^*}{s_0} \right| &  \leq& C\cdot s_0\\
        \end{array}
    \end{equation*}
hold for all $t\in [0,\infty)$. Thus, informally speaking, the approximation errors of $s$ by 
$s^*$, and of $c$ by $c^*$, are of order $s_0^2$.
\end{enumerate}
\end{proposition}

\begin{proof}
The first three items follow by straightforward calculations. As for part (d), the first column of \eqref{niceests} 
is just a restatement of {\bf Lemma \ref{estlem}}, while the second follows from the observation that (mutatis 
mutandis) the right-hand side of \eqref{eqmmlin} also obeys the estimates in parts (b) and (c) of this Lemma.

There remains part (e). In view of part (d) it suffices to show that
\begin{equation*}
    \begin{array}{rcl}
        \left|\dfrac{s_{\rm up}-s_{\rm low}}{s_0} \right| &  \leq& C\cdot s_0\\
            \left|\dfrac{c_{\rm up}-c_{\rm low}}{s_0} \right| &  \leq& C\cdot s_0\\
    \end{array}
\end{equation*}
for some constant and, in turn, it suffices to show such estimates for both $s_{\rm up}-s^*$, $s^*-s_{\rm low}$, 
and $c_{\rm up}-c^*$, $c^*-c_{\rm low}$. We will outline the relevant steps, pars pro toto, for the upper estimates.

By \eqref{nicelinsol} and part (c), one may write
\[
\dfrac{1}{s_0}
    \begin{pmatrix}
       s_{\rm up}\\ c_{\rm up}
    \end{pmatrix}
    =B_1(s_0,\pi)\exp (-A_1(s_0,\pi)t)+B_2(s_0,\pi)\exp (-A_2(s_0,\pi)t),
\]
with the $B_i$ and $A_i$ continuously differentiable in a neighborhood of $[0,K^*/2]\times \mathcal{M}\times[0,\infty)$.
Moreover, one sees
\[
    \dfrac{1}{s_0}
    \begin{pmatrix}
       s^*\\ c^*
    \end{pmatrix}
    =B_1(0,\pi)\exp (-A_1(0,\pi)t)+B_2(0,\pi)\exp (-A_2(0,\pi)t).
\]
It suffices to show that (with the maximum norm $\Vert \cdot \Vert$)
\[
\Vert B_i(s_0,\pi)\exp (-A_i(s_0,\pi)t)-B_i(0,\pi)\exp (-A_i(0,\pi)t)\Vert \leq s_0\cdot\,{\rm const.}
\]
in $\widetilde {\mathcal M}:=[0,K^*/2]\times \mathcal{M}\times[0,\infty)$, for $i=1,\,2$.

From compactness of $[0,K^*/2]\times {\mathcal M}$, continuous differentiability and the explicit form of $A_i$ and 
$B_i$ one obtains estimates, with the parameters and variables in $\widetilde{\mathcal M}$:
\begin{itemize}
    \item There is $A_*>0$ so that $|A_i(s_0,\pi)|\geq A_*$. 
    \item There is $B^*>0$ so that $\Vert B_i(s_0, \pi)\Vert \leq B^*$. 
    \item There are continuous $\widetilde B_i$ so that $B_i(s_0,\pi)-B_i(0,\pi)=s_0\cdot \widetilde B_i(s_0,\pi)$ 
    (Taylor), hence there is $B^{**}>0$ so that $\Vert B_i(s_0,\pi)-B_i(0,\pi)\Vert \leq s_0\cdot B^{**}$. 
    \item By the mean value theorem there exists $\sigma$ between $0$ and $s_0$ so that 
    \[
    \exp\left(-A_i(s_0, \pi)t\right)-\exp\left(-A_i(0, \pi)t\right)= -t\cdot \dfrac{\partial A_i}{\partial s_0}\,(\sigma,\pi)\cdot \exp\left(-A_i(\sigma, \pi)t\right)\,s_0.
    \]
    So, with the constant $A^{**}$ satisfying $A^{**}\geq |\dfrac{\partial A_i}{\partial s_0}|$ for all arguments 
    in $\widetilde{\mathcal{M}}$ one gets 
    \[
    |\exp\left(-A_i(s_0, \pi)t\right)-\exp\left(-A_i(0, \pi)t\right)|\leq A^{**}\cdot t\exp\left(-A_i(\sigma, \pi)t\right)\, s_0\leq A^{**}\cdot t\exp\left(-A_*t\right)\leq \dfrac{A^{**}}{A_*}s_0.
    \]
\end{itemize}
So
\begin{equation*}
    \begin{array}{cl}
         &   \Vert B_i(s_0,\pi)\exp (-A_i(s_0,\pi)t)-B_i(0,\pi)\exp (-A_i(0,\pi)t)\Vert \\
       \leq   &  \Vert B_i(s_0,\pi)-B_i(0,\pi)\Vert \cdot \exp (-A_i(s_0,\pi)t)\\
        & + \Vert B_i(0,\pi)\Vert \cdot |\exp (-A_i(s_0,\pi)t)-\exp (-A_i(0,\pi)t)| \\
        \leq &s_0\cdot (B^{**}+ B^*\cdot A^{**}/A_*),\\
    \end{array}
\end{equation*}
and this proves the assertion. 
\end{proof}
For illustrative purposes, a numerical example is given in Figure \ref{FIG1}.

\begin{remark}
Our result should not be mistaken for the familiar fact that in the vicinity of the stationary 
point $0$ the Michaelis--Menten system~\eqref{eqmmirrev} is (smoothly) equivalent to its linearization~\eqref{eqmmlin}; 
see, e.g. Sternberg~\cite[Theorem 2]{Stern}. The point of the Proposition is that the estimate, and the convergence 
statement, hold globally.
\end{remark}

\begin{remark}Some readers may prefer a dimensionless parameter (and symbols like $\ll$) to gauge 
the accuracy of the approximation. In view of {\bf Lemma \ref{estlem}} and equation~\eqref{star} it seems natural 
to choose $\varepsilon:=s_0/K_S<s_0/K_M$, which estimates the convergence of $G$ and $H$ to $Df(0)$, and set up the 
criterion $s_0/K_S\ll 1$. But still the limit $s_0\to 0$ has to be considered.
\end{remark}

\begin{remark}
Since explicit expressions are obtainable from \eqref{nicelinsol} for the $A_i$ and $B_i$ in the 
proof of part (e), one could refine it to obtain quantitative estimates. This will not be pursued further here.
\end{remark}

\subsection{Timescales for the linear Michaelis--Menten system~\eqref{eqmmlin}}\label{subs:time} 
In contrast to familiar reduction scenarios for the Michaelis--Menten system~\eqref{eqmmirrev}, low initial substrate 
linearization~\eqref{eqmmlin} does not automatically imply a separation of timescales. Timescale separation depends 
on further conditions that we discuss next.

Recall the constants introduced in \eqref{constants} and note
\[
\Delta= k_1^2\left((K_M+e_0)^2-4Ke_0\right).
\]
In this notation, the eigenvalues of the matrix in the linearized system \eqref{eqmmlin} are 
\begin{equation}\label{evaldisc}
    \lambda_{1,2}=\frac{k_1}{2}(K_M+e_0)\left(-1\pm\sqrt{1-\dfrac{4Ke_0}{(K_M+e_0)^2}}\right).
\end{equation}
This matrix is the Jacobian at the stationary point $0$, thus it is possible to use the results from earlier 
work \cite{ESWAnti}.

\begin{remark}
Since timescales are inverse absolute eigenvalues, by \cite[subsection 3.3]{ESWAnti}, \eqref{evaldisc} shows:
\begin{enumerate}[(a)]
\item The timescales are about equal whenever
\[
4Ke_0\approx (K_M+e_0)^2,
\]
which is the case, notably, when $K_S\approx 0$ (so $K_M\approx K$) and $K_M\approx e_0$.
\item   On the other hand, one sees from \eqref{evaldisc} that a significant timescale separation exists whenever 
(loosely speaking, employing a widely used symbol)
\begin{equation}\label{timescalecond}
    \dfrac{4Ke_0}{(K_M+e_0)^2}\ll 1.
\end{equation}
Notably this is the case when $e_0/K_M$ is small, or when $K/K_M$ is small. 
\item Moreover, condition~\eqref{timescalecond} is satisfied in a large part of parameter space (in a 
well-defined sense) as shown in Eilertsen et al.~\cite[subsection 3.3, in particular Figure~2]{ESWAnti}.
\end{enumerate}
\end{remark}

\begin{proposition}\label{timescaleprop}
We take a closer look at the scenario with significant timescale separation, stating the results in a somewhat 
loose language.
\begin{enumerate}[(a)]
\item Given significant timescale separation as in \eqref{timescalecond}, the slow eigenvalue is approximated by
    \begin{equation}
        \lambda_1\approx -k_1\cdot\dfrac{Ke_0}{K_M+e_0}=-\dfrac{k_2e_0}{K_M+e_0},
    \end{equation}
and the reduced equation is 
    \begin{equation}
        \dfrac{d}{dt}\begin{pmatrix} s\\ c
        \end{pmatrix} =-\dfrac{k_2e_0}{K_M+e_0}\begin{pmatrix} s\\ c
        \end{pmatrix}.
    \end{equation}
\item Given significant timescale separation as in \eqref{timescalecond}, the eigenspace for the slow eigenvalue
is asymptotic to the subspace spanned by
    \begin{equation}
       \widehat v_1:=\begin{pmatrix}
                 K_M-\dfrac{Ke_0}{K_M+e_0}\\ e_0
             \end{pmatrix}.
    \end{equation}
\end{enumerate}
\end{proposition}

\begin{proof}
Given \eqref{timescalecond} one has 
\[
\sqrt{(K_M+e_0)^2-4Ke_0}=(K_M+e_0)\sqrt{1-\dfrac{4Ke_0}{(K_M+e_0)^2}}\approx(K_M+e_0)\left(1-\dfrac{2Ke_0}{(K_M+e_0)^2}\right)
\]
by Taylor approximation, from which part (a) follows.

Moreover, according to \eqref{nicelinsol}, the direction of the slow eigenspace is given by 
\[
v_1^*=\begin{pmatrix}
        (K_M-e_0)+\sqrt{(K_M -e_0)^2+4K_Se_0}\\2e_0
    \end{pmatrix} =\begin{pmatrix}
        (K_M-e_0)+\sqrt{(K_M +e_0)^2-4Ke_0}\\2e_0
    \end{pmatrix}.
\]
By the same token, we have the approximation 
\begin{equation}
    \begin{array}{rcc}
        v_1^*    &\approx&
        \begin{pmatrix}
        (K_M-e_0)+(K_M+e_0)\left(1-\dfrac{2e_0K}{(K_M+e_0)^2}\right)\\2e_0
        \end{pmatrix}  \\
            & =&2
        \begin{pmatrix}
                K_M-\dfrac{Ke_0}{K_M+e_0}\\ e_0
        \end{pmatrix} 
        =2 \widehat v_1,
    \end{array}
\end{equation}
which shows part (b).
\end{proof}

\begin{remark}
One should perhaps emphasize that {\bf Proposition~\ref{timescaleprop}} indeed describes a singular perturbation reduction when
$\lambda_1\to 0$. This can be verified from the general reduction formula in Goeke and Walcher~\cite{gw2} (note that the
Michaelis--Menten system~\eqref{eqmmlin} is not in standard form with separated slow and fast variables), with elementary 
computations.
\end{remark}
    
Finally we note that {\bf Proposition~\ref{timescaleprop}} is applicable to the discussion of the linearized total 
quasi-steady-state approximation. Details are discussed in \cite{ESWUnreason}.

\section{Conclusion}
Our comprehensive analysis of the Michaelis--Menten system under low initial substrate concentration addresses an important 
gap in the literature. We obtain the sufficient\footnote{Necessity can easily be verified. See also the discussion in \cite{ESWUnreason} for examples.} condition, $s_0 \ll K$ for the global approximation of the Michaelis--Menten 
system~\eqref{eqmmirrev} by the linear Michaelis--Menten system~\eqref{eqmmlin}. This simplification, known as the pseudo-first-order 
approximation, is widely used in transient kinetics experiments to measure enzymatic reaction rates. Moreover, given the 
biochemical significance of this reaction mechanism and the renewed interest in the application of the total quasi-steady 
state approximation in pharmacokinetics experiments under low initial substrate concentration \cite{BYKK,VSTetal}, the 
linearization analysis holds clear practical relevance.  

From a mathematical perspective, it is intriguing to observe a simplification that arises independently of timescale separation 
or invariant manifolds.  The natural next step is to explore this approach with other familiar enzyme reaction mechanism 
(e.g., those involving competitive and uncompetitive inhibition, or cooperativity [see, for example, Keener and 
Sneyd~\cite{KeSn}]).  However, a direct application of Kamke's comparison theorem is not always feasible. Enzyme catalyzed 
reaction mechanism with competitive or uncompetitive inhibition do not yield cooperative differential equations, and 
standard cooperative networks do so only within specific parameter ranges. Investigating these systems under low initial 
substrate conditions will require the development of further mathematical tools. 

\section*{Data availability}
We do not analyse or generate any datasets, because our work proceeds within a theoretical and mathematical approach.


\end{document}